\documentclass[11pt]{amsart}

\usepackage{enumerate}
\usepackage{amsfonts}
\usepackage{amssymb}
\usepackage{amsmath}
\usepackage{tikz}
\usepackage{pgf}
\usepackage{cite}
\usetikzlibrary{arrows,topaths,intersections,calc}
\usetikzlibrary{decorations.markings}
\usepackage{stmaryrd}
\usepackage{filecontents}
\usepackage{cleveref}
\usepackage[mathscr]{euscript}

\newtheorem{theorem}{Theorem}[section]

\newtheorem{lemma}[theorem]{Lemma}
\newtheorem{proposition}[theorem]{Proposition}
\newtheorem{corollary}[theorem]{Corollary}

\theoremstyle{definition}

\theoremstyle{remark}
\newtheorem{remark}[theorem]{Remark}
\newtheorem*{ack}{Acknowledgements}

\numberwithin{equation}{section}

\DeclareMathOperator{\Gal}{Gal}

\DeclareMathOperator{\Hom}{Hom}

\newcommand{\mfrak}[1]{\mathfrak{#1}}

\newcommand{\mbb}[1]{\mathbb{#1}}

\newcommand{\ZP}{\mbb{Z}_p}

\begin{document}

\title[Stickelberger elements and $p$-adic $L$-functions]{On Stickelberger Elements for $\mbb{Q}(\zeta_{p^{n+1}})^+$ and $p$-adic $L$-functions}


\author{Timothy All}
\email{timothy.all@rose-hulman.edu}
\address{5500 Wabash Ave, Terre Haute, IN 47803}



{}


\date{}

\dedicatory{For my parents}

\begin{abstract}
We give a survey of a couple known constructions of $p$-adic $L$-functions including Iwasawa's construction from classical Stickelberger elements. We then construct ``real'' Stickelberger elements, i.e., explicit elements in the Galois group ring with $\ZP$ coefficients that annihilate the Sylow $p$-subgroup of the ideal class group of $\mathbb{Q}(\zeta_{p^{n+1}})^+$. In analogy with Iwasawa's work, we show that these elements are coherent in $\ZP$-towers and give rise to twisted $p$-adic $L$-functions.
\end{abstract}

\maketitle

\section{Introduction}

Let $\chi$ be a Dirichlet character. The $L$-function attached to $\chi$ is defined on $\Re(s) >1$ by
\[ L(s,\chi) = \sum_{n=1}^{\infty} \frac{\chi(n)}{n^s}\]
It would be an understatement to say that these functions play an important role in algebraic number theory. To illustrate what we mean, let $k$ be an abelian number field with Galois group $G_k$. We write $\widehat{G}_k = \Hom_{\mbb{Z}} (G_k, \mbb{C})$. By the Kronecker-Weber theorem, $k \subset \mbb{Q}(\zeta_m)$ for some minimal positive integer $m$, so we may associate $\widehat{G}_k$ with a certain subgroup of Dirichlet characters on $(\mbb{Z}/m\mbb{Z})^{\times}$. Let $h_k$ denote the class number of $k$, $R(k)$ the regulator of $k$, $w_k$ the order of the group of roots of unity of $k$, and $d_k$ the discriminant of $k$. We have the following classical result:
\begin{equation} \label{classnoformula} \frac{2^{r_1} (2\pi)^{r_2} h_k R(k)}{w_k \sqrt{|d_k|}} = \prod_{\substack{\chi \in \widehat{G}_k \\ \chi \neq 1}} L(1,\chi),\end{equation}
where $r_1$ (resp. $r_2$) is the number of real places (resp. pairs of complex places) of $k$. The above formula is a first indication that the $L$-functions attached to characters in $\widehat{G}_k$ encode fundamental algebraic properties of $k$.

Of particular interest is the class number $h_k$. This is because many Diophantine equations can be attacked by converting to an ideal equation in the ring of integers of an appropriate number field. The classical example is the Fermat equation: let $p$ be an odd prime (from here on out) and consider $x^p + y^p =z^p$. Assuming a non-trivial solution exists, we obtain the ideal equation
\[ \prod_{j=0}^{p-1} (x+ y\zeta_p^j ) = (z)^p \]
in $\mbb{Z}[\zeta_p]$ where $\zeta_p = e^{2\pi i/p}$. Under some additional assumptions, the ideals $(x+ y\zeta_p^j)$ can be shown to be co-prime, thus $(x+\zeta_p y) = \mfrak{a}^p$ for some ideal $\mfrak{a} \subset \mbb{Z}[\zeta_p]$. If $p \nmid h_k$, then this implies that $\mfrak{a}=(\alpha)$ is principal. On the other hand, a relation of the sort $x+y\zeta_p = \epsilon \alpha^p$ for a unit $\epsilon$ is impossible. This type of approach can be applied to many Diophantine equations, and the quality of $h_k \bmod{p}$ is often a feature of those arguments.

In order to better understand the quality of $h_k \bmod{p}$ or, even better, $h_k \bmod{p^n}$, it makes sense  considering \Cref{classnoformula} to search for a $p$-adic analogue of $L(s,\chi)$. In particular, we want a continuous function $L_p(\cdot, \chi) : \mbb{Z}_p \to \mbb{C}_p$ such that
\begin{equation} \label{define} L_p(1-n, \chi) = (1- \chi(p) p^{n-1}) L(1-n, \chi) \footnote{The ``Euler factor'' at $p$ must be removed for issues relating to convergence.} \end{equation}
for every integer $n > 0$ satisfying $n\equiv 1 \bmod{p-1}$. Such a function is necessarily unique. Moreover, if $\chi$ is odd then $L_p(s,\chi)$ is identically zero. This is because
\[ L(1-n,\chi) = -B_{n,\chi}/n, \]
for all integers $n \geq 1$ where the $B_{n,\chi}$ are the generalized Bernoulli numbers. Specifically, the rational numbers $B_{n,\chi}$ are defined by the following formula:
\begin{equation} \label{genbernoulli} \sum_{a=1}^{f_{\chi}} \frac{\chi(a) t e^{at}}{e^{f_{\chi}t} -1} = \sum_{n=0}^{\infty} B_{n,\chi} \frac{t^n}{n!},\end{equation}
where $f_{\chi}$ is the conductor of $\chi$. But $B_{n,\chi} = 0$ in the event that $n \equiv 1 \bmod{p-1}$ and $\chi$ is odd. Since $L_p(1-n,\chi)$ vanishes on a dense subset of $\mbb{Z}_p$, it follows that $L_p(s,\chi)$ is identically zero in this case.

So it would seem $p$-adic $L$-functions are naturally suited to comment on totally real number fields. In fact, we have the following $p$-adic analogue of \Cref{classnoformula}. For a totally real abelian number field $k$ of degree $n$ over $\mbb{Q}$, we have
\begin{equation} \label{padicclassnoformula} \frac{2^{n-1} h_k R_p(k)}{\sqrt{d_k}} = \prod_{\substack{\chi \in \widehat{G} \\ \chi \neq 1}} \left( 1 - \frac{\chi(p)}{p} \right)^{-1} L_p(1,\chi), \end{equation}
where $R_p(k)$ is the $p$-adic Regulator of $k$. The quantity $R_p(k)$ is defined just as $R(k)$ but with the Iwasawa $p$-adic logarithm replacing the usual logarithm. Note that the terms $R_p(k)$ and $\sqrt{d_k}$ are only defined up to a sign, but $R_p(k)/\sqrt{d_k}$ can be defined without ambiguity. It is known that $R_p(k)$ is non-zero for abelian fields and conjectured to be non-zero for all number fields $k$, but a proof has yet to be found.

Various constructions of $p$-adic $L$-functions have been given over the years, the first of which was offered by Leopoldt and Kubota \cite{LK}. Specifically, they show that the function
\[ \frac{1}{s-1} \lim_{n \to \infty} \frac{1}{p^n[f,p]} \sum_{\substack{x=1 \\ (x,p)=1}}^{p^n[f,p]} \chi(x) \sum_{n=0}^{\infty} \binom{1-s}{n} ( \langle x \rangle -1)^n,\]
where $[f,p]$ is the least common multiple of $f$ and $p$, in fact, satisfies \Cref{define}. Iwasawa \cite{I} famously constructed $L_p(s,\chi)$ from the classical Stickelberger elements in the theory of cyclotomic fields. As a result of his construction, he showed that there exists a power series $F(T) \in \mbb{Z}_p(\chi) \llbracket T-1 \rrbracket$ such that $F((1+p)^s) = L_p(s,\chi)$ where $\ZP(\chi)$ is the ring of integers of $\mbb{Q}_p(\chi(1),\chi(2),\ldots)$. This result is not only very deep by the nature of its construction, but it is also very useful.

It is a theme of this article that constructions of $p$-adic $L$-functions, and other related objects, benefit from a certain measure theoretic point of view. For example, the Leopoldt-Kubota construction revolves around expressing $L_p(s,\chi)$ as the $\Gamma$-transform of a measure arising from a certain rational function related to \Cref{genbernoulli}. In \S 2, we outline the notation and setup to appreciate this measure theoretic viewpoint.

In \S 3, we discuss Iwasawa's construction using odd-parts of the classical Stickelberger elements. In \S 4, we discuss a new construction that mirrors that of Iwasawa but takes place in the ``plus-part''. In particular, we construct real analogs of Stickelberger elements, i.e., explicit elements in a Galois group ring with $\ZP$-coefficients that annihilate the Sylow $p$-subgroup of the ideal class group of $\mbb{Q}(\zeta_{p^{n+1}})^+$. The Kummer-Vandiver conjecture states that $p \nmid h_{k_n^+}$, so these elements are of particular interest. We then show that the even-parts of these new Stickelberger elements give rise to twisted $p$-adic $L$-functions. In \S 5, we summarize and compare results across the previous sections. Our main result is \Cref{main}.

\section{Preliminaries}

Let $\mfrak{o}$ be the ring of integers of $K$, a finite extension of $\mbb{Q}_p$. Let $\Gamma$ be a group topologically isomorphic to $\mbb{Z}_p$. We write $\Gamma_n = \Gamma/\Gamma^{p^n}$, this gives us $\Gamma = \varprojlim \Gamma_n$ under the natural maps. Let $\gamma$ be a fixed topological generator for $\Gamma$, and write $\gamma_n$ for $\gamma \bmod{\Gamma^{p^n}}$. For any commutative ring $R$, we set
\begin{align*}
R\llbracket \Gamma \rrbracket &:= \varprojlim R[\Gamma_n], \\
R\llbracket T-1 \rrbracket &:= \text{the power series ring with coefficients in $R$}.\\
\end{align*}
An $R$-valued \emph{distribution} on $\Gamma$ is a collection of maps $\{ \mu_n : \Gamma_n \to R \}$ with the following property:
\[ \mu_n(x) = \sum_{y \mapsto x} \mu_{n+1}(y).\]
We typically write $\mu(a+p^n\ZP)$ in place of $\mu_n(\gamma_n^a)$. Let $R(\Gamma)$ denote the ring (under convolution) of $R$-valued distributions. For a continuous function $f:\mbb{Z}_p \to \mbb{C}_p$ and a distribution $\mu \in \mbb{C}_p(\Gamma)$, we write
\[ \int_{\ZP} f\ \mathrm{d}\mu := \lim_{n \to \infty} \sum_{x=0}^{p^n-1} f(x)\mu(x + p^n\ZP),\]
if it exists. The integral is guaranteed to exist if there exists an integer $M$ such that $|\mu(x+p^n\ZP)|_p < M$ for all $x$ and $n$. In this case, we say $\mu$ is \emph{bounded}. We write $F_{\mu}(T)$ for the \emph{Fourier transform} of $\mu$:
\[ F_{\mu}(T) = \int_{\ZP} T^x\ \mathrm{d}\mu(x) = \sum_{n=0}^{\infty} \left( \int_{\ZP} \binom{x}{n}\ \mathrm{d}\mu(x) \right) (T-1)^n \in \mbb{C}_p\llbracket T-1 \rrbracket.\]
If $F \in \mfrak{o} \llbracket T-1 \rrbracket$, then there exists a distribution $\mu_F$ whose Fourier transform is $F$, specifically
\[ \mu_F(a+p^n \ZP) = \frac{1}{p^n} \sum_{x=0}^{p^n-1} \zeta_{p^n}^{-ax} F(\zeta_{p^n}^x).\]
Finally, for $\mu \in R(\Gamma)$, let $\theta_{\mu}$ denote the following element of $R[\Gamma_n]$:
\[ \theta_{\mu} = \left( \sum_{a=0}^{p^n-1} \mu(a+p^n\ZP) \gamma_n^{-a} \right).\]
The map $R(\Gamma) \to R\llbracket \Gamma \rrbracket$ defined by $\mu \to \theta_{\mu}$ is an isomorphism of rings. Even better, when $R = \mfrak{o}$, the following isomorphisms fit into a commutative diagram:
\begin{center}
\begin{tikzpicture}[scale=.75]
\node (M) at (-1.5,0) {$\mfrak{o}(\Gamma)$};
\node (L) at (1,0) {$\mfrak{o}\llbracket T-1 \rrbracket$};
\node (GR) at (1,-2) {$\mfrak{o}\llbracket \Gamma \rrbracket$};

\node (m) at (6,0) {$\mu$};
\node (Fm) at (8,0) {$F_{\mu}$};
\node (thm) at (8,-2) {$\theta_{\mu}$ };

\node (mF) at (10,0) {$\mu_F$};
\node (F) at (12,0) {$F$};
\node (thF) at (12, -2) {$\theta_F$};

\node (text) at (3.5,-1) {given by};

\node (text) at (9,-1) {or};

\draw [->,-latex] (M) to node {} (L);
\draw [->,-latex] (L) to node {} (GR);
\draw [->,-latex] (M) to node {} (GR);

\draw [|->] (m) to node {} (Fm);
\draw [|->] (Fm) to node {} (thm);
\draw [|->] (m) to node {} (thm);

\draw [|->] (mF) to node {} (F);
\draw [|->] (F) to node {} (thF);
\draw [|->] (mF) to node {} (thF);
\end{tikzpicture}
\end{center}
In this setup, $\gamma^{-1} \in \mfrak{o} \llbracket \Gamma \rrbracket$ corresponds to the point-mass distribution centered at $1$. Accordingly, we have $\gamma^{-1}$ corresponds to $T$ in $\mfrak{o} \llbracket T-1 \rrbracket$, so $\gamma$ corresponds to $1/T \in \mfrak{o}\llbracket T-1 \rrbracket$.

Alternatively, we have $\mfrak{o} \llbracket \Gamma \rrbracket \simeq \mfrak{o}\llbracket X \rrbracket$ induced by the Iwasawa isomorphism
\begin{align*}
\mfrak{o}[\Gamma_n] &\to \mfrak{o}\llbracket X \rrbracket /(1-(1+X)^{p^n}) \\
\gamma_n &\mapsto 1+X \bmod{1-(1+X)^{p^n}}.
\end{align*}
If $\mu$ is the $\mfrak{o}$-valued distribution underlying $\theta_{\mu} \in \mfrak{o}\llbracket \Gamma \rrbracket$, we write $I_{\mu}(X) \in \mfrak{o}\llbracket X \rrbracket$ to denote the corresponding power series under the Iwasawa isomorphism. In this setup, $\gamma$ corresponds to $1+X$. Altogether, we have $\mfrak{o} \llbracket \Gamma \rrbracket \simeq \mfrak{o} \llbracket X \rrbracket \simeq \mfrak{o} \llbracket T-1 \rrbracket$ via $\gamma \mapsto 1+X \mapsto T$, so
\[ I_{\mu}(X) = F_{\mu} \left( \frac{1}{X+1} \right) \quad \text{and} \quad F_{\mu}(T) = I_{\mu} \left( \frac{1}{T} -1 \right).\]

Any $x \in \ZP^{\times}$ can be written as a product of a $p-1$-st root of unity $\omega(x)$ and an element $\langle x \rangle \in 1+ p\ZP$. We naturally think of $\omega$ as a Dirichlet character of conductor $p$. For a distribution $\mu \in \mfrak{o}(\Gamma)$, we define the \emph{Gamma transform} of $\mu$, denoted $G_{\mu}(s)$, by
\[ G_{\mu}(s) = \int_{\ZP^{\times}} \langle x \rangle^s\ \mathrm{d}\mu(x).\]

As mentioned before, this measure-theoretic language facilitates the types of interpolation problems encountered when constructing $L_p(s,\chi)$. For example, let $\mfrak{o}$ be the ring of integers of a finite extension of $\ZP$ containing the values of the Dirichlet character $\chi$. Let $c >1$ be an integer co-prime to $p$, and let $\chi_c$ be the function defined by
\[ \chi_c(a) = \begin{cases}
    \chi(a),& c \nmid a \\
    \chi(a)(1-c),& c \mid a.
    \end{cases}\]
Let $F(T)$ be the rational function
\[ F(T) = \frac{ \sum_{a=1}^{f_{\chi}} \chi_c(a) T^a}{1 - T^{f_{\chi}}}.\]
One can show that $F(T) \in \mfrak{o}\llbracket T-1 \rrbracket$ (by virtue of our definition for $\chi_c$), so let $\mu_{F} \in \mfrak{o}(\Gamma)$ be the corresponding distribution. By differentiating the defining formula for $F$, we see that for non-negative integers $k$
\[ \int_{\ZP} x^k\ \mathrm{d}\mu_F(x) = \left. \left( T \frac{\mathrm{d}}{\mathrm{d}T} \right)^k F(T) \right|_{T=1}.\]
In particular, this gives us
\[ \int_{\ZP} x^k\ \mathrm{d}\mu_F(x) = \left.\frac{\mathrm{d}^k}{\mathrm{d}z^k}  F(e^z) \right|_{z=0} = L(-k,\chi_c).\]
The last equality in the above follows from the fact that the coefficient on $z^k/k!$ in the Laurent expansion of $F(e^z)$ at $z=0$ is equal to $L(-k,\chi_c)$ where $L(s,\chi_c)$ is the analytic continuation of the Dirichlet series defined for $\Re(s) >1$ by
\[ L(s,\chi_c) = \sum_{n=1}^{\infty} \frac{\chi_c(n)}{n^s}.\]
From the above expression, we get that
\[ L(-k,\chi_c) = (1-\chi(c) c^{k+1}) L(-k,\chi).\]
The Gamma transform $G_{\mu_F}(s)$ is an interpolation of the integrals of $x^k$ with respect to $\mu_F$. Taking the limit over integers $k$ such that $k \to s$ ($p$-adically) and $k \equiv 0 \bmod{p-1}$, we see that
\[ \lim \int_{\ZP} x^k\ \mathrm{d}\mu_F(x) = \int_{\ZP^{\times}} \langle x \rangle^s\ \mathrm{d}\mu_F(x).\]
Whence
\begin{equation} \label{Lpgamma} \frac{G_{\mu_F}(s)}{1-\chi(c) \langle c \rangle^{s+1}} = L_p(-s,\chi).\end{equation}
The fact that the $p$-adic $L$-function is the Gamma transform of a rational function was a crucial ingredient in Sinnott's proof \cite{Sgamma} of the vanishing of the $\mu$-invariant for the cyclotomic $\ZP$-extension of an abelian number field.
\section{The Minus Side}

From \Cref{Lpgamma}, one can deduce that there are elements $G,H \in \mfrak{o}\llbracket T-1 \rrbracket$ such that
\[ L_p(-s,\chi) = \frac{G(\kappa^s)}{H(\kappa^s)},\]
where $\kappa$ is a fixed generator for $1 + p\ZP$. The fact that this is true was originally proven by Iwasawa \cite{I}, though he proves even more. In particular, if $\chi$ (resp. $\psi$) is a tame (resp. wild) character, then there exists $G_{\chi}, H_{\chi} \in \mfrak{o}\llbracket T-1 \rrbracket$ such that
\[ L_p(s,\chi\psi) = \frac{G_{\chi}(\zeta_{\psi} \kappa^s)}{H_{\chi}(\zeta_{\psi} \kappa^s)} \]
where $\zeta_{\psi} = \psi^{-1}(\kappa)$. This is a very important result for a number of reasons. For example, the right hand side of \Cref{padicclassnoformula} can now be expressed with a handful of power series. This allows one to study the growth of the class number of $k_n$ and $k_n^+$ as $n$ grows (see \cite{IW2}).

For clarity of exposition, we restrict ourselves to the following special case. Let $k_n = \mbb{Q}(\zeta_{p^{n+1}})$. We have the decomposition
\[ W \times U =\ZP^{\times} \]
where $W$ is the set of $p-1$-st roots of unity in $\ZP$ and $U = 1 + p\ZP$. We write $\sigma_a$ to denote the automorphism $\zeta_{p^{n+1}} \mapsto \zeta_{p^{n+1}}^a$ in $G_n = \Gal(k_n/\mbb{Q})$. For any $x \in \ZP$, we write $s_n(x)$ to denote the unique integer in the interval $[0,p^{n+1})$ satisfying $s_n(x) \equiv x \bmod{p^{n+1}}$. Then $s_n$ provides an onto morphism
\begin{align*}
\ZP^{\times} &\to G_n \\
x &\mapsto \sigma_{s_n(x)}.
\end{align*}
Let $\gamma_n = \sigma_{s_n(\kappa)} \in \Gal(k_n/k_0) = \Gamma_n$, so we have
\[ \zeta_{p^{n+1}}^{\sigma_{s_n(\zeta \kappa^a)}} = \zeta_{p^{n+1}}^{\zeta \kappa^a} = \zeta_{p^{n+1}}^{\sigma_{s_n(\zeta)} \gamma_n^a}.\]

\subsection{Stickelberger Elements}
Let $\ell$ be a rational prime congruent to $1 \bmod{p^{n+1}}$. Let $\tau$ be a generator for $\Gal(k_n(\zeta_{\ell})/k_n)$, and consider the Gauss sum
\[ g_n(\ell) = - \sum_{b=1}^{\ell-1} \zeta_{p^{n+1}}^b \zeta_{\ell}^{\tau^b}.\]
If $\lambda$ is a prime of $k_n$ above $\ell$, then for some integer $c$ with $(c,p)=1$, one shows that $g_n(\ell)^{c \sigma_c^{-1} -1} \in k_n$ and
\[ \big( g_n(\ell)^{c\sigma_c^{-1}-1} \big) = \lambda^{(c- \sigma_c) \theta_n},\]
where $\theta_n$ is the group ring element
\[ \frac{1}{p^{n+1}} \sum_{\substack{a=1 \\ (a,p)=1}}^{p^{n+1}} a \sigma_a^{-1} \in \mbb{Q}[\Gal(k_n/\mbb{Q})]. \]
The factorization of $g_n(\ell)$ was originally of interest because of problems pertaining to reciprocity, but upon further inspection it's quite peculiar. After all, if we change the prime $\ell$, the group ring element $\theta_n$ still remains! This gestalt shift in viewing the factorization of the Gauss sums $g_n(\ell)$ leads us to
\begin{theorem}[Stickelberger] \label{classicStickelberger}
For any $\beta \in \mbb{Z}[G]$ such that $\beta \theta_n \in \mbb{Z}[G]$, (for instance, let $c \in \mbb{Z}$ such that $(c,p)=1$ and take $\beta =(c-\sigma_c)$), we have that $\beta \theta_n$ annihilates the ideal class group of $k_n$.
\end{theorem}
This result allows us to investigate the class group of $k_n$ using an element that only depends on the conductor of $k_n$. See \cite{SStickelberger} for more information on Stickelberger elements in the ``minus-part''.

\subsection{$p$-adic $L$-functions}

Let $\chi$ be a non-trivial character of $W \simeq \Gal(k_0/\mbb{Q})$, and let $e_{\chi}$ denote the idempotent
\[ \frac{1}{|W|} \sum_{\zeta \in W} \chi(\zeta) \sigma_{s_n(\zeta)}^{-1}.\]
Let $\theta_n(\chi) \in \mbb{Q}_p[\Gamma_n]$ denote the $\chi$-component of the Stickelberger element $\theta_n$: $e_{\chi} \theta_n = \theta_n(\chi) e_{\chi}$. We have
\[ \theta_n(\chi) = \sum_{a=0}^{p^n-1} \left( \frac{-1}{p^{n+1}} \sum_{\zeta \in W} s_n(\zeta \kappa^a) \chi(\zeta)^{-1} \right) \gamma_n^{-a}.\]
\begin{lemma}[Iwasawa] The elements $\theta_n(\chi)$ are coherent, i.e., the sequence $(\theta_n(\chi))$ belongs to $\mbb{Q}_p\llbracket \Gamma \rrbracket$.
\end{lemma}
\begin{proof}
The elements $\theta_n(\chi)$ are coherent if and only if the map
\[ \mu_{\theta}^{\chi}(a+p^n \ZP) = \frac{-1}{p^{n+1}} \sum_{\zeta \in W} s_n(\zeta \kappa^a) \chi(\zeta)^{-1} \]
forms a distribution. Note that
\[ \sum_{i=0}^{p-1} \mu_{\theta}^{\chi}(a+ ip^n + p^{n+1}\ZP) = \frac{-1}{p^{n+2}} \sum_{\zeta \in W} \left( \sum_{i=0}^{p-1} s_{n+1}(\zeta \kappa^{a+ip^n}) \right) \chi(\zeta)^{-1},\]
while
\[ s_{n+1}(\kappa^{a+ip^n}\zeta) = s_n(\zeta \kappa^a) + s_0(x_{n+1} + i y_0)  \]
where $x_{n+1}$ (resp. $y_0$) is the $n+1$-st (resp. $0$-th) coordinate in the $p$-adic expansion of $\zeta \kappa^a$ (resp. $\zeta$). As $i$ varies between $0$ and $p-1$, so does $s_0(x_{n+1} + i y_0)$, hence
\[ \sum_{i=0}^{p-1} s_{n+1}(\zeta \kappa^{a+ip^n}) = \frac{p(p-1)}{2}.\]
Since $\chi \neq 1$, it follows that
\[ \sum_{i=0}^{p-1} \mu_{\theta}^{\chi}(a + i p^n + p^{n+1}\ZP) = \mu_{\theta}^{\chi}(a+ p^n \ZP).\]
This completes the proof of the lemma.
\end{proof}

\begin{remark} If $\chi \neq 1$ is even, then
\[ \mu_{\theta}^{\chi}(a+ p^n\ZP) = \frac{-1}{p^{n+1}} \sum_{\zeta \in W} s_n(-\zeta \kappa^a) \chi(\zeta)^{-1}.\]
But $s_n(-x) = p^{n+1} - s_n(x)$ which implies $\mu_{\theta}^{\chi}(a+p^n \ZP) = - \mu_{\theta}^{\chi}(a+p^n\ZP)$. So if $\chi$ is even, then $\mu_{\theta}^{\chi}(a+p^n \ZP) =0$ for all $a$. This shows us that in order to obtain interesting results, we must take $\chi$ to be an \emph{odd} character. Therefore, we assume $\chi$ is odd for the remainder of this section.
\end{remark}

Now that we have a coherent sequence $\theta_n(\chi)$ and a matching distribution $\mu_{\theta}^{\chi}$ we'd like to know what analytic functions come with them. But unfortunately, $\mu_{\theta}^{\chi}$ is generally an unbounded distribution. On the other hand, this distribution was born out of Stickelberger elements and we know how to \emph{integralize} them. To boot, the \emph{integralizers} $1 - c\sigma_c^{-1}$ are, in fact, coherent themselves! This is a crucial observation.

Let $c$ be an integer not equal to $\pm 1$. Let $\theta_n(\chi_c)$ denote $\chi$-component of $(1-c\sigma_c^{-1}) \theta_n$:
\[ \theta_n(\chi_c) = \sum_{a=0}^{p^n-1} \left( \frac{-1}{p^{n+1}} \sum_{\zeta \in W} r_n(\zeta \kappa^a) \chi(\zeta)^{-1} \right) \gamma_n^{-a} \]
where for every $x \in \ZP$
\[ r_n(x) = \frac{s_n(x) - s_n(xc^{-1}) c}{p^{n+1}} \in \mbb{Z}.\]
Using the Fourier transform and the fact that under this map $\gamma_n^{-a} \mapsto T^a \bmod{1 - T^{p^n}}$, we get
\begin{align*}
F_{\theta}^{\chi_c}(\kappa^{m-1}) &\equiv - \sum_{a=0}^{p^n-1} \sum_{\zeta \in W} r_n(\zeta \kappa^a) \chi(\zeta)^{-1} \kappa^{a(m-1)} \bmod{p^{n+1}} \\
&\equiv - \sum_{\substack{b=1 \\ (b,p)=1}}^{p^{n+1}} b^{m-1} r_n(b) \cdot \chi^* \omega^{-m}(b)
\end{align*}
where $\chi \chi^* = \omega$. The last line follows from the fact that $\zeta \kappa^{a(m-1)} = \omega^{-m+1}(\zeta) (\zeta \kappa^a)^{m-1}$. Notice the change in parity; $\chi^*$ is an even character. The idea here is to try to take advantage of the following fact
\[ \lim_{n \to \infty} \frac{1}{p^{n+1}} \sum_{\substack{b=1\\(b,p)=1}}^{p^{n+1}} b^m \cdot \chi^* \omega^{-m}(b) = (1-\chi^*\omega^{-m}(p) p^{m-1}) B_{m,\chi^*\omega^{-m}}.\]
We use the congruence
\[ b^{m-1} r_n(b) \cdot \chi^* \omega^{-m}(b) \equiv \frac{1}{m p^{n+1}} \left( s_n(b)^m - s_n(bc^{-1})^mc^m \right) \bmod{p^{n+1}} \]
to obtain
\[ F_{\theta}^{\chi_c} (\kappa^{m-1}) \equiv \frac{-(1- c^m \chi^*\omega^{-m}(c))}{m} \cdot \frac{1}{p^{n+1}} \sum_{\substack{b=1\\ (b,p)=1}}^{p^{n+1}} b^m \cdot \chi^* \omega^{-m}(b) \bmod{p^{n+1}}.\]
Taking limits we get
\[ \frac{F_{\theta}^{\chi_c}( \kappa^{m-1} ) }{1 - c^m \cdot \chi^* \omega^{-m}(c)} = L_p(1-m,\chi^*).\]
Choose $c$ so that $\langle c \rangle = \kappa$. Then we have
\begin{theorem}[Iwasawa \cite{I}] Let $L_{\theta^-}^{\chi}(T)$ be the function defined by
\[ L_{\theta^-}^{\chi}(T) = \frac{F_{\theta}^{\chi_c}(T)}{1 - \kappa c\chi^{-1}(c) \cdot T}.\]
Then $L_{\theta^-}^{\chi}(\kappa^s) = L_p(-s,\chi^*)$.
\end{theorem}

\begin{remark} So the $p$-adic $L$-function attached to $\chi^*$ (an even character) arises from the inverse limit of the $\chi$-parts (an odd character) of the classical Stickelberger elements $\theta_n$. Note that we can choose $c$ so that $1 - \kappa c \chi^{-1}(c) \cdot T$ is invertible (in $\ZP\llbracket T-1 \rrbracket$) if and only if $\chi \neq \omega$. So if $\chi \neq \omega$ (i.e., $\chi^* \neq 1$), we have $L_{\theta^-}^{\chi}(T) \in \ZP\llbracket T-1 \rrbracket$ and the isomorphisms
\begin{equation} \label{functional1} \frac{\ZP(\Gamma)}{(\mu_{\theta}^{\chi})} \simeq \frac{\ZP\llbracket \Gamma \rrbracket}{\big( \theta(\chi) \big)} \simeq \frac{\ZP\llbracket T-1 \rrbracket}{\big( L_{\theta^-}^{\chi}(T) \big)}.\end{equation}
This is a very satisfying result. Given \Cref{classnoformula,padicclassnoformula} and a host of other instances where $L$-functions appear to comment on algebraic structure, one wonders whether it's all just a cosmic coincidence or if it's a symptom of a deeper connection. This theorem points us towards the latter. In lieu of \Cref{classicStickelberger}, one might even dare to hope that $L_{\theta^-}^{\chi}(T)$ is essentially the characteristic of the $\ZP\llbracket \Gamma \rrbracket \simeq \ZP \llbracket T-1 \rrbracket$-module $e_{\chi} \varprojlim A_n$ where $A_n$ denotes the $p$-part of the class group of $k_n$. Miraculously, this is all true and proven by Mazur and Wiles \cite{MW}.

Since $L_{\theta^-}^{\chi} (T)$ annihilates $e_{\chi} \varprojlim A_n$, it follows that the twisted power series $L_{\theta^-}^{\chi} \left( \frac{1}{\kappa T} \right)$ annihilates $e_{\chi^*} \varprojlim A_n^+$ where $A_n^+$ denotes the $p$-part of the class group of $k_n^+$. Here's why: consider the $\chi^*$-part of $\mathscr{X}_{\infty}$, the Galois group of the maximal abelian pro-$p$ extension of $\mbb{Q}(\zeta_{p^{\infty}})^+$ unramified outside $p$. Note that $e_{\chi^*} \mathscr{X}_{\infty}$ contains $e_{\chi^*} \varprojlim A_n^+$ as a quotient, what's more, $e_{\chi^*} \mathscr{X}_{\infty}$ is isomorphic to the Kummer dual of the $\chi$-part of $\varinjlim A_n$. But this latter module is essentially a twisted version of $e_{\chi} \varprojlim A_n$ from which we may derive the claimed annihilation.

In the next section, we construct \emph{explicit} annihilators of $A_n^+$ for every level $n \geq 0$ in the vein of the classical Stickelberger theorem. We begin with an explicit element that could be considered a ``real'' Gauss sum, and we explain how the factorization of this element gives rise to annihilators of $A_n^+$. These are new and build upon our work in \cite{A}. We then show that these annihilators can be induced to give rise to the aforementioned twisted $p$-adic $L$-functions in a manner that parallels Iwasawa's work. This is interesting in its own right, but also indicates that these annihilators are, in a sense, full bodied.
\end{remark}

\section{The Plus Side}

Let $k_n^+$ denote the maximal real subfield of $k_n$, and associate $\Gamma_n$ with $\Gal(k_n^+/k_0^+)$. For every positive integer $n$, we let
\[ \delta_n(T) = \frac{\zeta_{p^{n+1}}^{g\kappa} -T}{\zeta_{p^{n+1}} -T} \]
where $g$ is a primitive root mod $p$. We write $\delta_n$ for $\delta_n(1)$, $K_n$ for $\mbb{Q}_p(\zeta_{p^{n+1}})$, and $\mfrak{o}_n$ for $\mbb{Z}_p[\zeta_{p^{n+1}}]$, the valuation integers of $K_n$.

\subsection{Stickelberger Elements} Let $\ell$ be a rational prime congruent to $1 \bmod{p^{n+1}}$, and let $\tau$ be a generator for $\Gal(k_n^+(\zeta_{\ell})/k_n^+)$. Consider the ``Gauss sum''
\[ g_n^+(\ell) = - \sum_{b=1}^{\ell -1} \delta_n(\zeta_{\ell})^{N_b} \zeta_{\ell}^{\tau^b} \]
where
\[ N_b = 1 + \tau + \cdots + \tau^{b-1} \in \mbb{Z}[\Gal(k_n^+(\zeta_{\ell})/k_n^+)].\]
The following lemma explains why we call this a Gauss sum.
\begin{lemma} The element $g_n^+(\ell)$ is non-zero for some choice of $\zeta_{\ell}$. What's more, we have $\delta_n(\zeta_{\ell}) \cdot g_n^+(\ell)^{\tau} = g_n^+(\ell)$.
\end{lemma}

\begin{proof} Let $\zeta_{\ell} = e^{2\pi i/\ell}$. Let $B$ be the $(\ell-1)\times(\ell-1)$-matrix
\[ B = \begin{bmatrix} 1 & 1 & \cdots & 1 \\ \zeta_{\ell}^{\tau} & \zeta_{\ell}^{\tau^2} & \cdots & \zeta_{\ell}^{\tau^{\ell-2}} \\ \vdots & \vdots & \ddots & \vdots \\ \zeta_{\ell}^{\ell-2} & \zeta_{\ell}^{(\ell-2)\tau} & \cdots & \zeta_{\ell}^{(\ell-2)\tau^{\ell-2}} \end{bmatrix} \]
Now let $A$ be the matrix $B$ with the first row removed. Since $\det B \neq 0$, it follows that $\dim \ker A = 1$. In fact, $\ker A$ is the span of the column vector
\[ \begin{bmatrix} 1 - \zeta_{\ell} \\ \vdots \\ 1 - \zeta_{\ell}^{\tau^{\ell-2}} \end{bmatrix},\ \text{so } A \cdot \begin{bmatrix} \delta_n(\zeta_{\ell})^{N_1} \\ \vdots \\ \delta_n(\zeta_{\ell})^{N_{\ell-1}} \end{bmatrix} = \vec{0}\ \text{  implies that  }\  \delta_n(\zeta_{\ell})= \frac{1-\zeta_{\ell}^{\tau}}{1-\zeta_{\ell}}.\]
Now, using the fact that $N_b \cdot \tau = N_{b+1} -1$, we see that
\[ g_n^+(\ell)^{\tau} = - \delta_n(\zeta_{\ell})^{-1} \sum_{c=2}^{\ell} \delta_n(\zeta_{\ell})^{N_c} \zeta_{\ell}^{\tau^c}.\]
Since $N_{\ell} = N + N_1$ where $N$ is the norm, and $\delta_n(\zeta_{\ell})^N =1$ (see \cite[Theorem 3.1]{A}) , our lemma follows.
\end{proof}

\begin{remark} Note that if one replaces $k_n^+$ with $k_n$ and $\delta_n(\zeta_{\ell})$ with $\zeta_{p^{n+1}}$, then we obtain the classical Gauss sum $g_n(\ell)$. \end{remark}

The fact that $(g_n^+(\ell))$ is what Hilbert would call an ``ambiguous'' ideal is what enables its factorization to illuminate ideal relations; much like the classical Gauss sum $g_n(\ell)$. Unlike the classical case, the part of the factorization of $g_n^+(\ell)$ that depends on $\ell$ is much more difficult to isolate.

Toward that end, let $\alpha: E_n^+ \to \ZP[\Gal(k_n^+/\mbb{Q})]$ be a Galois module map where $E_n^+$ denotes the units of $k_n^+$ and fix an ideal class $\mfrak{c}$ in the Sylow $p$-subgroup of $\mfrak{C}(k_n^+)$, the ideal class group. For every integer $m > n$, let $\alpha_m(x)$ denote the group ring element of $\mbb{Z}[\Gal(k_n^+/\mbb{Q})]$ arrived at by applying the map $s_m$ to the coefficients of $\alpha(x)$. Then there exists a rational integer $\ell \equiv 1 \bmod{p^m}$ and a prime ideal $\lambda \mid \ell$ of $k_n^+$, such that
\[ \big( N(g_n^+(\ell))^{\tilde{\alpha}} \big) = \lambda^{\alpha_m(\delta_n)} \mfrak{b}^{\ell-1} \]
where $\mfrak{b}$ is some integral ideal, $\lambda \in \mfrak{c}$, and $\tilde{\alpha}$ is, for lack of a better word, a ``fudge-factor''. It follows that $\alpha(\delta_n)$ annihilates $\mfrak{C}(k_n^+) \otimes \ZP$. In fact, it's shown in \cite[Theorem 3.8]{A} that if $\alpha: E_n^+ \to \mfrak{o}_n[\Gal(k_n^+)/\mbb{Q}]$ is a Galois module map, then $\alpha(\delta_n)$ annihilates $\mfrak{C}(k_n^+) \otimes \mfrak{o}_n$. These ideas find their beginnings in a paper of Thaine \cite{Th} (see also \cite{R}).

So all we need for a Stickelberger-esque theorem in this setting is an explicit map $\alpha$ to complete the recipe. Consider the map
\begin{align*}
\alpha: k_n^{\times} &\to K_n[\Gal(k_n/\mbb{Q})] \\
x &\mapsto \sum_{\substack{a=1 \\ (a,p)=1}}^{p^{n+1}} \log_p(x^{\sigma_a}) \sigma_a^{-1},
\end{align*}
where $\log_p$ is the Iwasawa logarithm. Let $W_n \subseteq E_n$ denote the roots of unity of $k_n$ contained in the units of $k_n$, respectively. Since $E_n = W_n E_n^+$, it follows that $\alpha(E_n) = \alpha(E_n^+)$. We let $\mfrak{o}_n[\Gal(k_n/\mbb{Q})]$ act on any $\mfrak{o}_n[\Gal(k_n^+/\mbb{Q})]$-module by restriction. Let $\Theta_n = \alpha(\delta_n)$. The element $\Theta_n$ is a handsome analog of the classical Stickelberger element:

\begin{theorem}[{\cite[Theorem 1.1]{A}}]
For any $\beta \in K_n[\Gal(k_n/\mbb{Q})]$ such that $\beta \cdot \alpha(E_n)=\beta \cdot \alpha(E_n^+)$ is contained in $\mfrak{o}_n[\Gal(k_n/\mbb{Q})]$, we have that $\beta \Theta_n$ annihilates $\mfrak{C}(k_n^+) \otimes \mfrak{o}_n$.
\end{theorem}

\subsection{$p$-adic $L$-functions} Now, let $\chi \neq 1$. Let $\Theta_n(\chi) \in K_n[\Gamma_n]$ denote the $\chi$-component of the real Stickelberger element $\Theta_n$. Note that
\[ \left. \Theta_n \right|_{k_n^+} = \sum_{a=0}^{p^n-1} \sum_{\zeta \in W/\pm 1} \log_p \big( \delta_n^{\sigma_{s_n(\zeta \kappa^a)}+ \sigma_{s_n(-\zeta \kappa^a)}} \big) \sigma_{s_n(\zeta)}^{-1} \gamma_n^{-a}.\]
So if $\chi$ is odd, then $\Theta_n(\chi) = 0$. So we can only get interesting results if $\chi$ is an \emph{even} character in this setting. Therefore we assume $\chi$ is even for the remainder of this section. In this case, we have
\[ \Theta_n(\chi) = (\chi(g)\gamma_n-1) \sum_{a=0}^{p^n-1} \left( \sum_{\zeta \in W} \log_p\big( 1-\zeta_{p^{n+1}}^{\zeta \kappa^a} \big) \chi^{-1}(\zeta) \right) \gamma_n^{-a}. \]

\begin{lemma} The elements $\Theta_n(\chi)$ are coherent, i.e., the sequence $(\Theta_n(\chi))$ belongs to $\mbb{C}_p\llbracket \Gamma \rrbracket$.
\end{lemma}
\begin{proof} The elements $\Theta_n(\chi)$ are coherent if and only if the map
\[ M(a+p^n\ZP) = \sum_{\zeta \in W} \log_p \big( 1 - \zeta_{p^{n+1}}^{\zeta \kappa^a} \big) \chi^{-1}(\zeta) \]
forms a distribution since the sequence  $(\chi(g) \gamma_n-1)$ is coherent. Since $\kappa^{a+ip^n} \equiv \kappa^i (1+ip^{n+1}) \bmod p^{n+2}$, we have
\begin{align*}
\sum_{i=0}^{p-1} M(a + ip^n + p^{n+1}\ZP) &= \sum_{\zeta \in W} \chi^{-1}(\zeta) \sum_{i=0}^{p-1} \log_p \big( 1 - \zeta_{p^{n+2}}^{\chi(\zeta) \kappa^{a+i p^n}} \big) \\
&= \sum_{\zeta \in W}  \chi^{-1}(\zeta) \log_p \prod_{i=0}^{p-1} \big( 1 - \zeta_{p^{n+2}}^{\chi(\zeta)\kappa^a} \zeta_p^{i\chi(\zeta)} \big).
\end{align*}
Applying the relation
\begin{equation} \label{difference} \prod_{i=0}^{p-1} \big( x - y \zeta_p^i \big) = x^p - y^p,\end{equation}
we get that $M$ is a distribution, as claimed.
\end{proof}

Now, $\Theta_n(\chi)$ is a coherent sequence and consequently we have a distribution $\mu_{\Theta}^{\chi}$, but as in the classical setting, $\mu_{\Theta}^{\chi}$ is an unbounded distribution. What we need now is a collection of elements $\beta_n(\chi) \in \mfrak{o}_n[\Gamma_n]$ that \emph{integralize} the elements $\Theta_n(\chi)$; something akin to $1- c\sigma_c^{-1}$. For a fixed $\eta \in W$, consider
\[ \upsilon_n(\chi) = \sum_{a=0}^{p^n-1} \left( \sum_{\zeta \in W} \log_p \big( \eta - \zeta_{p^{n+1}}^{\zeta \kappa^a}\big) \chi^{-1}(\zeta) \right) \gamma_n^{-a}.\]
Since $\eta^p = \eta$, the same proof that shows $\Theta_n(\chi)$ are coherent also shows that $\upsilon_n(\chi)$ are coherent. The element $\upsilon_n(\chi)$ has an inverse in $K_n[\Gamma_n]$, say $\beta_n(\chi)$, if and only if the $\psi$-part of $\upsilon_n(\chi)$ is non-zero for every $\psi \in \widehat{\Gamma}_n$. Specifically, let $e_{\psi}$ denote the idempotent
\[ e_{\psi} = \frac{1}{p^n} \sum_{a=0}^{p^n-1} \psi(\kappa^a) \gamma_n^{-a}.\]
Let $\upsilon_n(\chi\psi) \in K_n$ such that $e_{\psi} \upsilon_n(\chi) = \upsilon_n(\chi\psi) e_{\psi}$. We have
\[ \upsilon_n(\chi\psi) = \sum_{a=0}^{p^n-1} \mu_{\upsilon}^{\chi}(a+p^n\ZP) \psi^{-1}(\kappa^a) = \int_{\ZP} \zeta_{\psi}^x\ \mathrm{d}\mu_{\upsilon}^{\chi}(x), \]
where $\zeta_{\psi} = \psi^{-1}(\kappa)$. So if for all $\psi\in \widehat{\Gamma}_n$ we have $\upsilon(\chi\psi) \neq 0$, then $\beta_n(\chi)$ exists. In particular, we have
\[ \beta_n(\chi) =\sum_{\psi \in \widehat{\Gamma}_n} \frac{e_{\psi}}{\upsilon_n(\chi\psi)} = \sum_{a=0}^{p^n-1} \left( \frac{1}{p^{n}} \sum_{\psi \in \widehat{\Gamma}_n} \frac{\psi(\kappa^a)}{\int_{\ZP} \zeta_{\psi}^x\ \mathrm{d}\mu_{\upsilon}^{\chi}(x)} \right) \gamma_n^{-a}.\]
Let $\mu_{\beta}^{\chi}$ be the associated distribution. The convolution $\mu_{\beta}^{\chi} * \mu_{\Theta}^{\chi}$ evaluated at $a+p^n\ZP$ is
\[ \sum_{b=0}^{p^n-1} \mu_{\Theta}^{\chi}(a+p^n \ZP) \left( \frac{1}{p^n} \sum_{\psi \in \widehat{\Gamma}_n} \frac{\zeta_{\psi}^{b-a}}{\int_{\ZP} \zeta_{\psi}^x\ \mathrm{d}\mu_{\upsilon}^{\chi}} \right). \]
Recollecting terms we arrive at
\begin{equation} \label{explicitStickelberger} \beta_n(\chi) \Theta_n(\chi) = \sum_{a=0}^{p^n-1} \left( \frac{1}{p^n} \sum_{\psi \in \widehat{\Gamma}_n} \zeta_{\psi}^{-a} \cdot \dfrac{ \int_{\ZP} \zeta_{\psi}^x\ \mathrm{d}\mu_{\Theta}^{\chi}(x) }{ \int_{\ZP} \zeta_{\psi}^x\ \mathrm{d}\mu_{\upsilon}^{\chi}(x) } \right) \gamma_n^{-a}.\end{equation}
We write $\vartheta_n(\chi)$ for $\beta_n(\chi)\Theta_n(\chi)$ and $\mu_{\vartheta}^{\chi}$ for the associated distribution. For any choice of $\eta$ (including $\eta=1$) and any choice of $c \in \ZP^{\times}$, it's straightforward to verify that $\mu_{\vartheta}^{\chi}$ is inert under $\sigma_{c}$. So $\mu_{\vartheta}^{\chi}$ is $\mbb{Q}_p$-valued. In fact, as we shall see, $\mu_{\vartheta}^{\chi}$ is $\ZP$-valued so \Cref{explicitStickelberger} describes an explicit and coherent sequence of annihilators in $\ZP\llbracket \Gamma \rrbracket$ of $\mfrak{C}(k_n^+) \otimes \ZP$.

\begin{proposition} For some choice of $p-1$-st root of unity $\eta$, the distribution $\mu_{\vartheta}^{\chi}$ is $\ZP$-valued. \end{proposition}
\begin{proof} Let $U_n$ denote the principal units of $K_n$ and $C_n$ the topological closure of the principal circular units of $k_n$ in $U_n$. There is a choice of $\eta$ such that
\[ u_n = \left( \frac{\eta- \zeta_{p^{n+1}}}{\omega(\eta -1)}\right)^{e_{\chi}} \]
is a $\ZP[\Gamma_n]$ generator for $e_{\chi}U_n$ (see \cite{W,IW}). What's more, the element
\[ c_n= \left( \zeta_{p^{n+1}}^{(1-g\kappa)/2} \frac{\zeta_{p^{n+1}}^{g \kappa} -1}{\zeta_{p^{n+1}} -1} \right)^{(p-1)\cdot e_{\chi}} \]
is a $\ZP[\Gamma_n]$ generator for $e_{\chi} C_n$ (see \cite{SStickelberger} for more on circular units). The definition for $\alpha$ extends to a map $A: K_n^{\times} \to K_n[\Gamma_n]$ defined in exactly the same way. Since $A$ is a Galois map and $\log_p$ vanishes at roots of unity, it follows that $A(u_n) = \upsilon_n(\chi)$ and
\[ A(c_n) = (p-1) \Theta_n(\chi). \]
Let $\vartheta_n(\chi) \in \mbb{Z}_p[\Gamma_n]$ such that $u_n^{(p-1)\vartheta_n(\chi)} = c_n$. Then
\[ \vartheta_n(\chi) \upsilon_n(\chi) = \Theta_n(\chi).\]
Note that
\[ \int_{\ZP} \zeta_{\psi}^x\ \mathrm{d}\mu_{\Theta}^{\chi}(x) = (\chi(g) \zeta_{\psi}^{-1} -1) \cdot \sum_{\substack{a=1 \\ (a,p)=1}}^{p^{n+1}} \log_p \big( 1 - \zeta_{p^{n+1}}^a \big) \chi\psi(a)^{-1} \neq 0 \]
since $\chi \neq 1$ and $L_p(1,\chi\psi) \neq 0$. So it must be that
\[ \int_{\ZP} \zeta_{\psi}^x\ \mathrm{d}\mu_{\upsilon}^{\chi}(x) \neq 0, \]
Hence $\upsilon_n(\chi)$ is invertible in $K_n[\Gamma_n]$. This gives us
\[ \beta_n(\chi)\Theta_n(\chi) = \vartheta_n(\chi) \in \ZP[\Gamma_n], \]
as desired.
\end{proof}
As noted before, this gives us the following
\begin{corollary} \label{realstickelberger1} The sequence $\big(\vartheta_n(\chi) \big)$ in $\ZP\llbracket \Gamma_n \rrbracket$ is such that $\vartheta_n(\chi)$ annihilates $e_{\chi} \cdot \left( \mfrak{C}(k_n^+) \otimes \ZP \right)$.
\end{corollary}

Now, let $I_{\vartheta}^{\chi}$ denote the power series in $\ZP \llbracket X \rrbracket$ obtained from $\big( \vartheta_n(\chi) \big)$ through the Iwasawa transform. Iwasawa \cite{IW} (see also \cite[Theorem 13.56]{W}) showed that there exists $S(X) \in \ZP \llbracket X \rrbracket^{\times}$ such that
\[ \frac{I_{\vartheta}^{\chi}(\kappa^s-1)}{S(\kappa^s-1)} = L_p(1-s,\chi).\]
Let $(\varsigma_n) \in \ZP\llbracket \Gamma \rrbracket$ such that $I_{\varsigma}(X) = S(X)$. Then
\[ L_p(1-s,\chi) = \left. \frac{I_{\vartheta}^{\chi}(X)}{I_{\varsigma}(X)} \right|_{X=\kappa^s-1} = \left. \frac{F_{\vartheta}^{\chi} (T) }{F_{\varsigma} (T)} \right|_{T=\kappa^{-s}} \]

This gives us the following
\begin{theorem} Let $L_{\theta^+}^{\chi}(T)$ be the function defined by
\[ L_{\theta^+}^{\chi}(T) = \frac{F_{\vartheta}^{\chi} (T)}{F_{\varsigma}(T)}.\]
Then $L_{\theta^+}^{\chi}(\kappa^{-s}) = L_p(1-s,\chi)$.
\end{theorem}

\begin{remark} So the $p$-adic $L$-function attached to $\chi$ (an even character) arises from the inverse limit of the $\chi$-parts of the real Stickelberger elements $\vartheta_n$. As before, for $\chi \neq 1$, we have $L_{\theta^+}^{\chi}(T) \in \ZP \llbracket T-1 \rrbracket$ and the isomorphisms
\begin{equation} \label{functional2} \frac{\ZP(\Gamma)}{\big( \mu_{\vartheta}^{\chi} \big)} \simeq \frac{ \ZP\llbracket \Gamma \rrbracket }{ \big( \vartheta_n(\chi) \big)}  \simeq \frac{ \ZP \llbracket T-1 \rrbracket}{\big( L_{\theta^+}^{\chi}(T) \big)}.\end{equation}
\end{remark}

\section{Comparison}

Let $\chi$ be an odd character not equal to $\omega$, so $\chi^* \neq 1$. From \Cref{functional1,functional2} we get the following theorem.
\begin{theorem} \label{functional}
For all $s \in \ZP$, we have
\[ L_{\theta^-}^{\chi} \big( \kappa^{-s} \big) = L_p\big( s,\chi^* \big) = L_{\theta^+}^{\chi^*} \big( \kappa^{s-1} \big).\]
\end{theorem}

Consider the map $\iota: \ZP \llbracket T-1 \rrbracket \to \ZP \llbracket T-1 \rrbracket$ defined by
\[ \iota: F(T) \mapsto F \left( \frac{1}{\kappa T} \right). \]
The map $\iota$ is well-defined since $\frac{1}{\kappa T} \in 1 + (p,T-1)$, what's more, $\iota$ is an involution. Hence, $\iota$ is an automorphism of $\mbb{Z}_p\llbracket T-1 \rrbracket$. This gives us the following
\begin{corollary} We have the following equality of power series:
\[ L_{\theta^-}^{\chi} \left( \frac{1}{\kappa T} \right) = L_{\theta^+}^{\chi^*}(T) \in \ZP\llbracket T-1 \rrbracket.\]
\end{corollary}
\begin{proof}
This follows immediately from the fact that $L_{\theta^-}^{\chi}(1/(\kappa T)) - L_{\theta^+}^{\chi^*}(T)$ vanishes on a neighborhood of $1$ and the Weierstrass preparation theorem in $\ZP \llbracket T-1 \rrbracket$.
\end{proof}

The elements $\theta_n^+(\chi^*)$, the sequence of group ring elements corresponding to $L_{\theta^+}^{\chi^*}(T)$, are of interest for the following reason.

\begin{proposition} \label{realstickelberger2} For every non-negative integer $n$, the elements $\theta_n^+(\chi^*)$ annihilate $e_{\chi^*} \left( \mfrak{C}(k_n^+) \otimes \ZP \right)$.
\end{proposition}
\begin{proof} From \Cref{realstickelberger1}, we know that $\vartheta_n(\chi^*)$ annihilates $e_{\chi^*} \left( \mfrak{C}(k_n^+) \otimes \ZP \right)$. Since $F_{\varsigma}(T)$ is an invertible power series in $\ZP \llbracket T-1 \rrbracket$, it follows that $\varsigma = (\varsigma_n)$ is invertible in $\ZP \llbracket \Gamma \rrbracket$. So $\varsigma_n \in \ZP[\Gamma_n]$ is invertible. It follows that if $\vartheta_n(\chi^*)$ annihilates $e_{\chi^*} \left( \mfrak{C}(k_n^+) \otimes \ZP \right)$, then so must $\vartheta_n(\chi^*) \cdot \varsigma_n^{-1} = \theta_n^+(\chi^*)$.
\end{proof}

Using \Cref{functional}, we can derive an explicit expression for $\mu_{\theta^+}^{\chi^*}$ which, in turn, gives an explicit expression for $\vartheta_n^+(\chi)$. Before giving this formula, we describe a little notation. Let $\mathbf{1}_{a+p^n\ZP}(x)$ denote the indicator function for $a+p^n \ZP$; $\mathbf{1}_{a+p^n \ZP}(x) = 1$ if $x \in a+p^n\ZP$ and $0$ otherwise. For any distribution $\mu \in \ZP (\Gamma)$, and any continuous function $f$, we define
\[ \int_{a+p^n\ZP} f(x)\ \mathrm{d}\mu(x) = \int_{\ZP} \mathbf{1}_{a+p^n\ZP}(x) f(x)\ \mathrm{d}\mu(x).\]
For any constant $b \in \ZP$, the distribution $\mu \circ b$ is defined by
\[ (\mu \circ b) (a + p^n \ZP) = \mu(ab + p^n\ZP).\]
We write $\mathrm{d}\mu(bx)$ for $\mathrm{d}(\mu \circ b)(x)$.

\begin{theorem} The distribution $\mu_{\theta^+}^{\chi^*}$ is given by
\[ \mu_{\theta^+}^{\chi^*}(a+p^n \ZP) = \int_{a+ p^n \ZP} \kappa^x\ \mathrm{d}\mu_{\theta^-}^{\chi}(-x). \]
\end{theorem}
\begin{proof}
We compute the Fourier transform of the measure given in the theorem. Note that
\begin{align*}
\int_{\ZP} T^x\ \mathrm{d}\mu_{\theta^+}^{\chi^*}(x) &= \lim_{n \to \infty} \sum_{a=0}^{p^n-1} T^a  \int_{a+p^n\ZP} \kappa^x\ \mathrm{d}\mu_{\theta^-}^{\chi}(-x) \\
&= \lim_{n \to \infty} \sum_{a=0}^{p^n-1} (\kappa T)^a \int_{a+p^n\ZP} \kappa^{x-a}\ \mathrm{d}\mu_{\theta^-}^{\chi}(-x).
\end{align*}
Consider the integral in the expression above. Let $b \in a+p^n\ZP$ and write $b= a + c p^n $ where $c \in \ZP$. For any $c$, we have $\kappa^{cp^n} = 1 + c_n p^{n+1}$ for some $c_n \in \ZP$. It follows that
\[ \int_{a+p^n\ZP} \kappa^{x-a}\ \mathrm{d}\mu_{\theta^-}^{\chi}(-x) = \mu_{\theta^-}^{\chi}(-a+p^n\ZP) + y_n(a) p^{n+1} \]
for some $y_n(a) \in \ZP$. Hence
\begin{align*}
\int_{\ZP} T^x\ \mathrm{d}\mu_{\theta^+}^{\chi^*} &= \lim_{n \to \infty} \sum_{a=0}^{p^n-1} (\kappa T)^a \mu_{\theta^-}^{\chi^*}(-a+p^n\ZP) + p^{n+1} (\kappa T)^a y_n(a) \\
&= \lim_{n \to \infty} \sum_{a=0}^{p^n-1} (\kappa T)^{-a} \mu_{\theta^-}^{\chi}(a+p^n \ZP) \\
&= L_{\theta^-}^{\chi} \left( \frac{1}{\kappa T} \right).
\end{align*}
The theorem now follows from \Cref{functional}.
\end{proof}

\begin{remark} Written out explicitly we have
\[ \mu_{\theta^+}^{\chi^*}(a+p^n\ZP) = \lim_{m \to \infty} \sum_{c=0}^{p^{m-n}-1} \kappa^{a+cp^n} \left( \frac{-1}{p^{n+1}} \sum_{\zeta \in W} s_n(\zeta \kappa^{-a-cp^n}) \chi(\zeta)^{-1} \right).\]
\end{remark}

We summarize what we've shown in the following theorem.

\begin{theorem} \label{main} Let $\chi \neq \omega$ be an odd character and $\chi^* = \chi^{-1} \omega$. Then
\begin{align*}
e_{\chi} \left( \mfrak{C}(k_n) \otimes \ZP \right)^{\theta_n^-(\chi)} &= 0 \\
e_{\chi^*} \left( \mfrak{C}(k_n^+) \otimes \ZP \right)^{\theta_n^+(\chi^*)} &= 0 = e_{\chi^*} \left( \mfrak{C}(k_n^+) \otimes \ZP \right)^{\vartheta_n(\chi^*)}
\end{align*}
The annihilators above are explicitly described in the following table along with their isomorphic images in $\ZP\llbracket T-1 \rrbracket$ and $\ZP(\Gamma)$:
\[ \begin{array}{|c|c|c|}
\hline
 \ZP \llbracket T-1 \rrbracket & \ZP(\Gamma) & \ZP \llbracket \Gamma \rrbracket \\
\hline
 L_{\theta^-}^{\chi}(T) & \underbrace{\frac{-1}{p^{n+1}} \sum_{\zeta \in W} s_n(\zeta \kappa^a) \chi(\zeta)^{-1}}_{=\mu_{\theta^-}^{\chi}(a+p^n\ZP)} & \underbrace{ \sum_{a=0}^{p^n-1} \mu_{\theta^-}^{\chi}(a+p^n\ZP) \gamma_n^{-a}}_{=\theta_n^-(\chi)} \\
\hline
F_{\varsigma}(T) L_{\theta^+}^{\chi^*}(T) & \underbrace{\frac{1}{p^n} \sum_{\psi \in \widehat{\Gamma}_n} \zeta_{\psi}^{-a} \cdot \frac{ \int_{\ZP} \zeta_{\psi}^x\ \mathrm{d}\mu_{\Theta}^{\chi^*}(x) }{ \int_{\ZP} \zeta_{\psi}^x\ \mathrm{d}\mu_{\upsilon}^{\chi^*}(x) }}_{=\mu_{\vartheta}^{\chi}(a+p^n\ZP)} & \underbrace{\sum_{a=0}^{p^n-1} \mu_{\vartheta}^{\chi^*}(a+p^n\ZP) \gamma_n^{-a}}_{=\vartheta_n(\chi^*)}  \\
\hline
 L_{\theta^+}^{\chi^*}(T) & \underbrace{ \int_{a+p^n\ZP} \kappa^x\ \mathrm{d}\mu_{\theta^-}^{\chi}(-x)}_{\mu_{\theta^+}^{\chi^*}(a+p^n\ZP)} & \underbrace{\sum_{a=0}^{p^n-1} \mu_{\theta^+}^{\chi^*}(a+p^n\ZP) \gamma_n^{-a}}_{\theta_n^+(\chi^*)}  \\
\hline
\end{array} \]
Moreover, we have
\[ L_{\theta^-}^{\chi} \big( \kappa^{-s} \big) = L_p\big( s,\chi^* \big) = L_{\theta^+}^{\chi^*} \big( \kappa^{s-1} \big).\]
\end{theorem}

The expressions for $\theta_n^+(\chi^*)$ and $\vartheta_n(\chi^*)$ are interesting as they may shed some new light on the quality of $\mfrak{C}(k_n^+) \otimes \ZP$. Can $\theta_n^+(\chi^*)$ be made more explicit? And what about $\vartheta_n(\chi^*)$. In a recent paper \cite{A2}, Waller and the author showed that although the distributions $\mu_{\Theta}^{\chi^*}$ and $\mu_{\upsilon}^{\chi^*}$ are unbounded in value they nonetheless form $C^1(\ZP)$ functionals through \emph{Volkenborn} integration. In fact, the Fourier transform of $\mu_{\beta}^{\chi^*}$ interpolates the kind of Gauss sums that appear in formulas for $L_p(1,\chi^*)$. Can they be used to shed light on the nature of $\vartheta_n(\chi^*)$? We hope to address these questions in the future.

\begin{ack} Many thanks to the anonymous referee for their careful review and for numerous suggestions that have improved the quality of this paper. I also want to thank David Goss for his encouragement during the writing of this manuscript.
\end{ack}

\bibliographystyle{plain}
\bibliography{mybib}

\end{document}